\documentclass[a4paper]{svproc}
\usepackage{url}
\usepackage{amsmath}
\usepackage{amssymb}

\usepackage{graphicx}
\usepackage{subcaption}
\usepackage{wrapfig}

\begin{document}
\mainmatter              % start of a contribution
\title{Implementation of Time-Varying Controllers for a Nonholonomic Mobile Robot: \\ Experimental Studies}
\titlerunning{Implementation of Time-Varying Controllers for a Mobile Robot}  % abbreviated title (for running head)
\author{Alexander Zuyev\inst{1,3} \and Victoria Grushkovskaya\inst{2,3} \and
Sebastian Eisner\inst{2}
\authorrunning{Alexander Zuyev et al.} % abbreviated author list (for running head)
\tocauthor{Alexander Zuyev, Victoria Grushkovskaya, and Sebastian Eisner}
\institute{%Princeton University, Princeton NJ 08544, USA,\\
%\email{I.Ekeland@princeton.edu},\\ WWW home page:
%\texttt{http://users/\homedir iekeland/web/welcome.html}
%\and
Max Planck Institute for Dynamics of Complex Technical Systems\\ Sandtorstra{\ss}e 1, 39106 Magdeburg, Germany
\and
Department of Mathematics, University of Klagenfurt\\ Universit\"atstra{\ss}e 65--67, 9020 Klagenfurt, Austria
\and
Inst. of Applied Mathematics \& Mechanics, National Acad. of Sciences of Ukraine
}}

\maketitle              % typeset the title of the contribution

\begin{abstract}
We consider a kinematic model of a wheeled mobile robot controlled by translational and angular velocities. For this class of nonholonomic systems, a family of time-varying feedback controllers was proposed in our previous works using gradient flow approximation techniques. In the present study, these controllers are implemented on a TurtleBot3 Burger (TB3) mobile robot to provide experimental validation of the stabilization problem with oscillating input signals. In addition, the admissibility problem of a gradient flow is investigated to justify the construction of a Lyapunov function candidate. The presented experimental results demonstrate the possibility of stabilizing the reference position of the robot using feedback controls with practically acceptable parameters.
\end{abstract}
\section{Introduction}
Control design problems for nonholonomic systems have attracted considerable interest from analytical~\cite{Bloch15,ZuSIAM,pazderski2017waypoint,gao2024fas,zhang2026prescribed}, computationally-oriented~\cite{mondal2020stability,zhang2022new,wu2025evolutionary,ebrahimi2026impact}, and data-driven~\cite{wang2023formation,miao2025study} perspectives. A unified approach to solving the stabilization, motion planning, and obstacle avoidance problems was proposed in~\cite{GZ23}, based on approximating the gradient flow of a certain potential function. Such a potential function plays the role of a Lyapunov function candidate in the stabilization problem.
Although it was proven that this strategy guarantees convergence of trajectories to the steady state for any positive-definite Lyapunov function candidate satisfying additional regularity conditions, the problem of optimizing the transient behavior remains open. In particular, numerical simulations show chattering phenomena when sampling-based control strategies are implemented (cf.~\cite{GZ18}). This chattering arises from the oscillatory nature of these time-varying controllers. In a realistic setup with physical limitations, the input signals vary within velocity constraints, causing considerable chattering around non-admissible directions. As a result, the control system is forced to switch repeatedly, leading to zigzagging behavior.

In this work, we consider a variational formulation aimed at minimizing chattering through a subsequent approximation of gradient-based control techniques. The proposed approach is based on measuring the admissibility of a given gradient flow generated by a potential function when applied to a nonholonomic system over a specified domain of the state space. The admissibility measure of the resulting set of trajectories, or phase portrait, within this domain is characterized in terms of an integral measure that accounts for the magnitude of the gradient of the generating potential function.

Note that the convergence proofs previously presented in~\cite{GZ23} require a certain design parameter $\varepsilon>0$ to be sufficiently small, which leads to high frequencies of order $1/\varepsilon$ in the oscillatory components of the controllers. However, implementing such controllers in real nonholonomic systems requires constraining the frequency within a given bandwidth while simultaneously satisfying physical control constraints. The present paper therefore aims to provide a proof of concept for the practical applicability of oscillatory controllers that account for acceptable ranges of frequency parameters and control gains.

\section{Stabilizing Controls for the Unicycle Kinematics}
In this paper, we consider the well-known kinematic model of a unicycle rolling on the $(x_1,x_2)$-plane, described by the following system  (see, e.g.,~\cite{Bloch15}):
\begin{equation}\label{uni}
\dot x = u_1 f_1(x) + u_2 f_2(x),\quad x=(x_1,x_2,x_3)^\top \in {\mathbb R}^3,\; u=(u_1,u_2)^\top \in{\mathbb R}^2,
\end{equation}
where the vector fields $f_1,f_2: {\mathbb R}^3\to {\mathbb R}^3$ are defined as $f_1(x)=\big(\cos x_3,\sin x_3,0\big)^\top$, $f_2(x)=\big(0,0,1\big)^\top$.
%$$
%f_1(x) = \begin{pmatrix}\cos x_3 \\ \sin x_3 \\ 0 \end{pmatrix},\;
%f_2(x) = \begin{pmatrix}0 \\ 0 \\ 1 \end{pmatrix}.
%$$
Here, $(x_1,x_2)$ are the coordinates of the contact point of the wheel,
and $x_3$ is the angle between the plane of the wheel and the $x_1$-axis.
The controls $u_1$ and $u_2$ describe the translational and angular velocities of the unicycle, respectively.
To simplify the presentation, all variables of this control system are assumed to be dimensionless.
For the purpose of control design, we introduce the matrix ${\cal F}(x)$ composed of the vector fields $f_1(x)$, $f_2(x)$ and their Lie bracket
$[f_1,f_2](x)=\frac {\partial f_2(x)}{\partial x} f_1(x) - \frac {\partial f_1(x)}{\partial x} f_2(x)$, and compute ${\cal F}^{-1}(x)$:
{\small
\begin{equation}\label{Fmatrix}
{\cal F}(x) = (f_1,f_2,[f_1,f_2])(x) = \begin{pmatrix} \cos x_3 & 0 & \sin x_3 \\ \sin x_3 & 0 & - \cos x_3 \\ 0 & 1 & 0 \end{pmatrix},
{\cal F}^{-1}(x) = \begin{pmatrix} \cos x_3 & \sin x_3 & 0 \\ 0 & 0 & 1 \\ \sin x_3 & -\cos x_3 & 0 \end{pmatrix}.
\end{equation}}
The nonsingularity of ${\cal F}(x)$ implies that system~\eqref{uni} satisfies the Lie Algebra Rank Condition (LARC):
$
{\rm span} (f_1(x),f_2(x),[f_1,f_2](x)) = {\mathbb R}^3$ {for all} $x\in{\mathbb R}^3$.
Hence, by the Chow--Rashevsky theorem, the driftless control-affine system~\eqref{uni} is completely controllable in~${\mathbb R}^3$.

For a given real-valued potential function $V\in C^2({\mathbb R}^3)$, the approximate motion of system~\eqref{uni} in the direction of the negative gradient of $V(x)$
is generated by control functions $u^\omega(x,t)=(u_1^\omega(x, t),u_2^\omega(x, t))^\top$ of the following form~\cite{GZ23}:
\begin{equation}\label{velControls}
\begin{aligned}
& u_1^\omega(x, t) = a_1(x) + k_1\sqrt{\omega|a_{12}(x)|}\,\mathrm{sign}(a_{12}(x))\cos(\omega t), \\
& u_2^\omega(x, t) = a_2(x) + k_2\sqrt{\omega|a_{12}(x)|}\sin(\omega t),
\end{aligned}
\end{equation}
where  $\omega{ =}\frac{2\pi}{\varepsilon} $ is the control frequency with sufficiently small $\varepsilon>0$, the coefficients $k_1$, $k_2$ satisfy $k_1 k_2{ = }4$,
and the  functions $a_1(x)$, $a_2(x)$, $a_{12}(x)$ are defined as
\begin{equation}
    \label{a_formulas}
a(x)=\left(
a_1(x),
a_2(x),
a_{12}(x)
\right)^\top
=-\gamma \mathcal F^{-1}(x)\nabla V(x)^\top,
\end{equation}
where $\gamma>0$ is a gain parameter, $\mathcal F^{-1}(x)$ is given by~\eqref{Fmatrix}, and $\nabla V (x)=\left(\frac{\partial V(x)}{\partial x_1},\frac{\partial V(x)}{\partial x_2},\frac{\partial V(x)}{\partial x_3}\right)$ denotes the gradient of $V(x)$.

In~\cite{GZ23}, the particular choice $k_1 = k_2 = 2$ was adopted. Here, we introduce distinct coefficients to provide more flexibility in tuning the control parameters. In particular, as will be shown in Section~4, an appropriate selection of $k_1$ and $k_2$ allows adjustment of the amplitudes of the oscillatory inputs so that the translational and angular velocities remain within the admissible control bounds.
The formulas for the control components in~\eqref{a_formulas} can be written explicitly using~\eqref{Fmatrix}:% as follows:
\begin{equation}\label{aPot}
\begin{aligned}
& a_1(x) = -\gamma\left(\frac{\partial V(x)}{\partial x_1}\cos x_3 +\frac{\partial V(x)}{\partial x_2}\sin x_3\right), \,a_2(x) = -\gamma \frac{\partial V(x)}{\partial x_3} , \\
& a_{12}(x) = -\gamma\left(\frac{\partial V(x)}{\partial x_1}\sin x_3 - \frac{\partial V(x)}{\partial x_2}\cos x_3\right).
\end{aligned}
\end{equation}
As discussed in~\cite{ZG_CDC25}, there are two approaches to implementing the controls~\eqref{velControls}.
The first approach is based on the sampling concept and assumes that the control amplitude vector $a(x)$ is updated at each time instant $t_j=\varepsilon j$ for $j=0,1, ...$ . In this case, the solutions of the resulting closed-loop system are treated according to the {\em $\pi_\varepsilon$-definition}~\cite{ZuSIAM,GZ23} with the control:
\begin{equation}
    \label{u_sampling}
u = u^\omega(x(t_j),t)\quad \text{for}\;\; t\in[j\varepsilon,(j+1)\varepsilon), \;\;j=0,1,\dots \; .
\end{equation}
For a class of positive definite functions $V$,
convergence of the $\pi_\varepsilon$-solutions to the zero equilibrium of systems of the form~\eqref{uni} has been proved in~\cite{ZuSIAM,GZ23} under certain technical assumptions.

The second approach assumes that the control amplitudes evolve continuously in time, that is,
\begin{equation}
    \label{u_cont}
u =u^\omega(x(t),t),\quad t\ge 0,
\end{equation}
for system~\eqref{uni}.
In the sequel, we will refer to the solution definitions corresponding to~\eqref{u_sampling} and~\eqref{u_cont} as the {\em ``sampling''} and {\em ``continuous-time''} approaches, respectively.
In Section~4, both approaches will be implemented experimentally.

\section{Justification for the Lyapunov Function Candidate}

The method of gradient flow approximation, developed in~\cite{GZ23}, establishes the solvability of the stabilization problem for a broad class of nonholonomic systems,
including the unicycle model~\eqref{uni}, with any sufficiently regular positive definite potential function.
However, the problem of the optimal choice of such a potential function has not been addressed so far.
To formulate this type of problem, we assume that the state $x$ and the control $u$ of system~\eqref{uni} are constrained to closed domains $X\subset {\mathbb R}^3$ and $U\subset {\mathbb R}^2$, respectively, and that $(x,u)=(0,0)$ is an interior point of $D\times U$.

For a function $V:X\to \mathbb R$ of class $C^2$, consider the gradient flow system
\begin{equation}\label{Gflow}
\dot {\bar x}(t) = - \nabla V(\bar x(t))^\top,\quad x(t)\in X,
\end{equation}
where the transpose on the right-hand side is taken under the convention that $\nabla V(\bar x)$ is a row vector.
System~\eqref{Gflow} possesses (local) existence and uniqueness of solutions to the Cauchy problem due to the assumption $V\in C^2(X;{\mathbb R})$.
We call the gradient flow of $V$ {\em admissible for system~\eqref{uni} in $X$} if, for each solution $\bar x(t)\in X$ of~\eqref{Gflow} defined on $t\in [0,+\infty)$,
there exists a control $\bar u:[0,+\infty)\to U$ such that $x=\bar x(t)$ satisfies~\eqref{uni} with $u=\bar u(t)$ for all $t\in [0,+\infty)$.
We then introduce the following ``measure of admissibility'' of the gradient flow of $V$ for system~\eqref{uni} in $X$ with controls in $U$:
\begin{equation}\label{J_X}
J_X[V] = \frac{1}{\mu (X)} \int_X \frac{\inf_{u\in U}\|u_1 f_1(x) + u_2 f_2(x) + \nabla V(x) \|^q}{\|\nabla V(x)\|^q}dx.
\end{equation}
Here, $\|\cdot \|$ denotes the Euclidean norm in $\mathbb R^n$, $\mu (X)$ is the Lebesgue measure of $X$, and $q$ is a positive parameter.
Clearly, if the gradient flow of $V$ is admissible for system~\eqref{uni} and $\nabla V(x)\neq 0$ for almost all $x\in X$,
then the integral in~\eqref{J_X} is well-defined and $J_X[V]=0$.
Hence, the minimization of~$J_X$ within some class of potential functions can be considered as a receipt for improving the performance of a gradient-based control algorithm to ensure better tracking properties of the reference dynamics~\eqref{Gflow} by the closed-loop trajectories.

The numerator in the integrand of~\eqref{J_X} is computed for system~\eqref{uni} by the following technical lemma.
\begin{lemma}\label{lemma_rho}
For $x,p\in {\mathbb R}^3$, define
\begin{equation}\label{rho_def}
\rho(x,p) :=\inf_{u\in {\mathbb R}^2} \|u_1 f_1(x) + u_2 f_2(x) + p\|.
\end{equation}
Then
\begin{equation}\label{rho_form}
\rho(x,p) = |p_1\sin x_3 - p_2 \cos x_3|\quad \text{for all}\;\; x,p\in {\mathbb R}^3.
\end{equation}
\end{lemma}
\begin{proof}
The function $\rho(x,p)$, defined by~\eqref{rho_def}, is the length of the projection of the vector $-p\in{\mathbb R}^3$ onto the subspace orthogonal to ${\rm span}\{f_1(x),f_2(x)\}$. Thus,
\begin{equation}\label{cross_prod}
\rho(x,p) = \frac{|\left<p,f_1(x) \times f_2(x)\right>|}{\|f_1(x)\times f_2(x)\|},
\end{equation}
where $\left<\cdot,\cdot\right>$ and $f_1(x) \times f_2(x)$ denote the inner product and the cross product of vectors in ${\mathbb R}^3$, respectively.
Straightforward computations show that
$
f_1(x) \times f_2(x) = (\sin x_3,-\cos x_3,0)^\top,\; \|f_1(x)\times f_2(x)\|=1,
$
so~\eqref{rho_form} follows from~\eqref{cross_prod}.
\end{proof}

A natural way of defining an optimal Lyapunov function candidate for system~\eqref{uni} in $X$ can be as follows: {\em find a function $V\in C^2(X;{\mathbb R})$ that minimizes the cost~\eqref{J_X} under the constraint that $V$ is positive definite in $X$}.
This optimization problem is highly nonlinear and nonconvex, and it poses significant analytical and computational challenges.
In particular, the set of positive definite functions is nonconvex in any reasonable function space,
and the functional~\eqref{J_X} exhibits singular behavior at critical points of $V$
due to the presence of $\nabla V$.
Note that, in the case $U={\mathbb R}^2$, the functional $J_X[V]$ is invariant under multiplication of $V$ by a nonzero constant.
Therefore, a solution to the above optimization problem (if it exists) is not unique without additional constraints.
A comprehensive analysis of this general problem is beyond the scope of the present paper.
Instead, we estimate $J_X[V]$ using a selected family of positive definite quadratic forms.
%within some positive definite quadratic forms.

\begin{wraptable}{r}{0.35\textwidth} % r - справа, l - слева
  \centering
        \begin{tabular}{|c|c|c|c|} \hline
             $c_1$             & $c_2$ & $c_3$     & $J_X[V]$     \\ \hline
            1      & 1  & 1  & 0.3333  \\
            {\bf 2}       & {\bf 1} & {\bf 1} & {\bf 0.3056} \\
            0.5       & 1  & 1  & 0.3658  \\
            1       & 2 & 1  & 0.4716  \\
            {\bf 1}       & {\bf 0.5} & {\bf 1} & {\bf 0.2123}  \\
             {\bf 1}     & {\bf 1}  & {\bf 2}  & {\bf 0.2228}  \\
              1      & 1  & 0.5 & 0.4219 \\ \hline
        \end{tabular}
%        \vspace{-1ex}
            \caption{Values of $J_X[V]$.}\label{tab:V}
\end{wraptable}

To estimate possible improvements of the admissibility measure $J_X[V]$ in comparison with the standard Lyapunov function candidate $V_1(x)=\|x\|^2$,
we assign different scalings to the three state variables of system~\eqref{uni}. Thus, we compute the functional $J_X[V]$ for the quadratic form $V(x)=c_1 x_1^2+c_2 x_2^2+c_3 x_3^2$ under different choices of coefficients $c_i>0$.
The computational results are summarized in Table~\ref{tab:V} for the case $X= [-1,1]^3$, $U={\mathbb R}^2$, and $q=2$.
In Table~\ref{tab:V}, the rows corresponding to parameters $(c_1,c_2,c_3)$ for which $V$ ensures better performance
than $V_1$, in the sense that $J_X[V]<J_X[V_1]\approx 0.3333$, are highlighted in bold.
The data indicate that performance improves when $c_1$ and $c_3$  are increased and $c_2$  is decreased relative to
the initial choice $c_1=c_2=c_3=1$. This observation motivates introducing the following Lyapunov function candidate for system~\eqref{uni}:
\vskip-2ex
\begin{equation}\label{V_alpha}
V_\alpha(x) = \alpha (x_1^2 +x_3^2)+ \frac{x_2^2}{\alpha},\quad \alpha>1,
\end{equation}
%\vskip-2ex
where $\alpha$ is a parameter.
In particular, computing the cost~\eqref{J_X} for $V_\alpha$ with $\alpha=2,4,10$, we obtain:
$
J_X[V_2]\approx 0.1403,\;
J_X[V_4]\approx 0.0962,\;
J_X[V_{10}]\approx 0.0906
$.

The obtained sequence of decreasing cost values justifies the use of the Lyapunov function candidates~\eqref{V_alpha} for further experiments.
It should be emphasized that quadratic forms cannot serve as control-Lyapunov functions for systems of the form~\eqref{uni} due to Brockett's necessary stabilizability condition~\cite{Bro83,Zu99}. Indeed, Brockett's condition and Artstein's theorem~\cite{Artstein}  imply that the considered class of nonholonomic systems does not admit any continuously differentiable control-Lyapunov function.

\section{Experimental Implementation}
\vskip-2ex
In this section, we present the experimental results obtained by implementing the proposed theoretical framework on a real robotic platform. We evaluate two Lyapunov functions: the classical squared norm and the quadratic form~\eqref{V_alpha}. In addition, we compare two  control approaches~\eqref{u_sampling} and~\eqref{u_cont}.

\subsection{Experimental Setup}
\vskip-1ex
\textbf{Robotic Platform.} The experiments were conducted using a {\em TurtleBot3 Burger (TB3)} – an affordable differential-drive mobile robot (see Figure~\ref{fig:TB3}, left)~\cite{robotis_TB3_overview}.
The robot's control and communication were handled by its onboard Raspberry Pi 4B, running Ubuntu Server 22.04.5 LTS (64-bit) and {\em ROS 2 Humble}. Mechanically, the TB3 features two off-centered, independently actuated wheels and a caster ball; thus, the robot is not a perfect unicycle.
However, a differential-drive robot can be modeled as a unicycle by mapping the velocities of the left and right wheels, $v_l$ and $v_r$
(linear velocities at the wheel edge), to the unicycle velocities  $u_1$ and $u_2$ introduced in~\eqref{uni}.
Let $d$ be the distance between the two wheels of the differential drive robot, and assume that the robot's center of rotation is at the midpoint between the wheels. Then, the unicycle's translational and angular velocities, respectively $u_1$ and $u_2$, are given by $
u_1 = \frac{v_r + v_l}{2}, \quad
u_2 = \frac{v_r - v_l}{d}.
$
The inverse of this transformation is handled automatically by the \texttt{diff\_drive\_controller} of ROS~2. It accepts unicycle-style \texttt{Twist} commands on \texttt{/cmd\_vel} and computes the corresponding wheel velocities.
The robot's motion is constrained by its kinematics with the maximum translational velocity of 0.22 m/s and the maximum angular velocity of 2.84 rad/s. No onboard sensors were used for state estimation or control, but four reflective markers were mounted for the motion capture system. \\
In practice, the TurtleBot3 Burger exhibits limitations beyond velocity constraints, including motor dead zones and nonlinearities. Although no encoder or motor velocity feedback was used to compensate for these effects, the experimental results show that the proposed controller still drives the system close to the desired state in a stable manner. Incorporating additional feedback into the control law could further improve the performance in more demanding scenarios and is left for future work.

\textbf{Motion Capture System.}
Robot pose information is obtained using an {\em OptiTrack} motion capture system, consisting of eight infrared cameras and the {\em Motive} tracking software. As previously stated, four reflective markers are attached to the top of the TB3, allowing the estimation of its $x$- and $y$-position as well as its yaw orientation.
The Motive software typically computes the center point of these markers and uses it as a reference pose.
However, this point
does not align precisely with the robot’s pivot point. To ensure accurate control,
the reference pose is adjusted manually to match the true center of rotation (see Figure~\ref{fig:TB3}, right).
We use a dedicated wireless access point to connect to the OptiTrack system via LAN, ensuring that the TB3 and a remote laptop used for robot control can communicate over a known, low-latency wireless network. \\
The robot moves on a standard office carpet floor, which minimizes slip. While the performance may vary on other surfaces, the conducted experiments demonstrate the feasibility of the controller. \\
Another important note is that the experiments rely on the OptiTrack system for high-precision position feedback. In environments without such a system, odometry-based approaches (i.e., relying only on IMU sensors and motor encoders) would be significantly more challenging.

\begin{figure*}[t]
    \centering
    \includegraphics[width=0.35\linewidth]{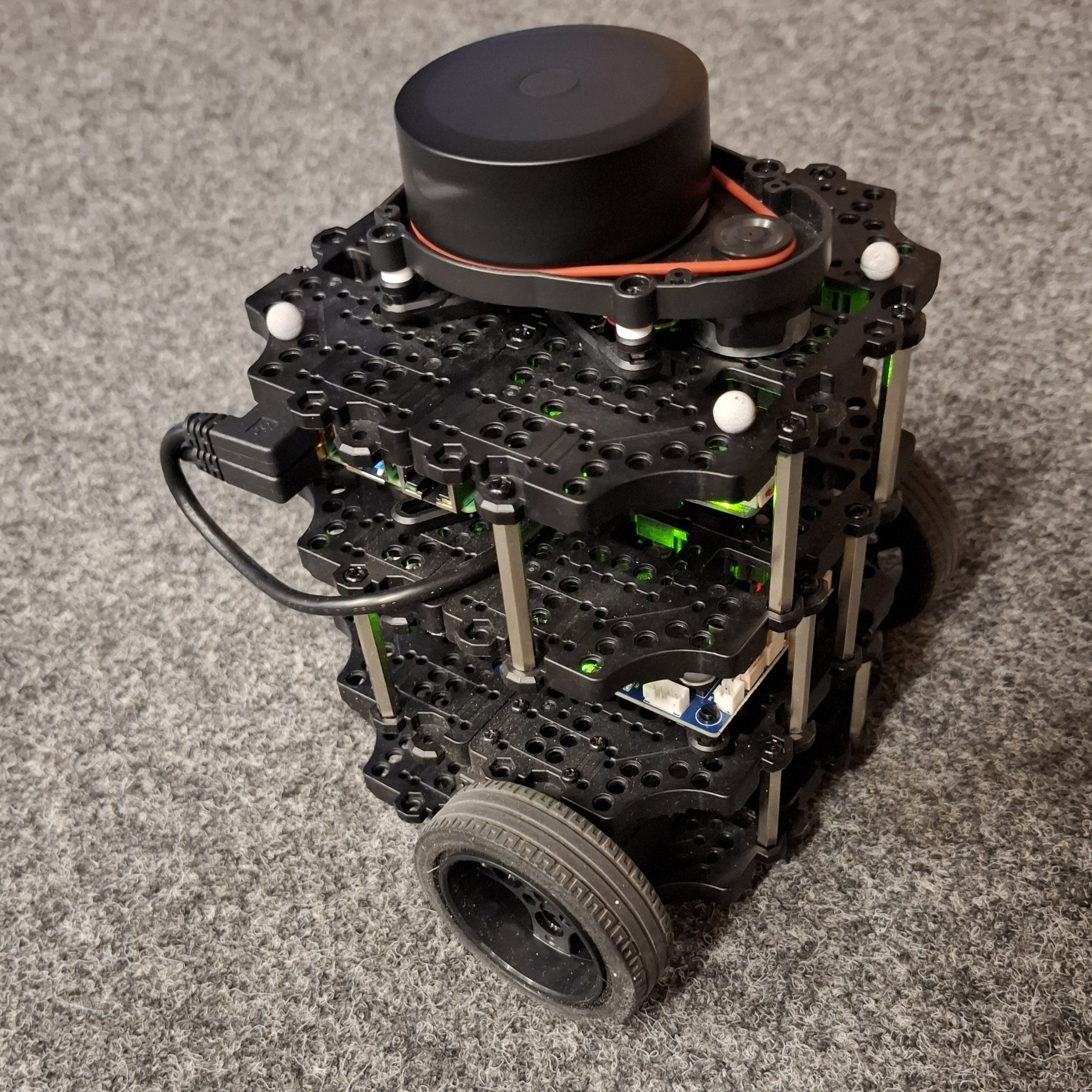} \qquad
    \includegraphics[width=0.35\linewidth]{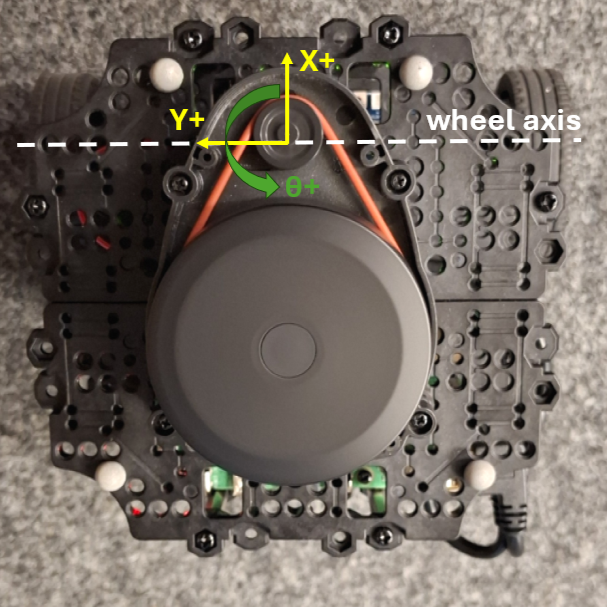}
    \caption{Left: TurtleBot3 Burger; Right: TB3 marker configuration}
    \label{fig:TB3}
\end{figure*}

\textbf{Software Architecture.}
The control logic for both simulation and real-world experiments is implemented as a single ROS 2 node using the \texttt{rclpy} Python library. In real-world experiments, the controller subscribes to the robot pose data provided by the OptiTrack system via the {\em NatNet} protocol, whereas in simulations it subscribes to the standard \texttt{/odom} ROS 2 topic published by {\em Gazebo}. Control commands are published on the standard \texttt{/cmd\_vel} topic, recognized by both the real-world TB3 and the Gazebo simulation. This allows us to use one and the same controller logic in different domains. \\
Although the physical dynamics of the robot are modeled in continuous time, a digital implementation of both controllers~\eqref{u_sampling} and~\eqref{u_cont} cannot be entirely continuous, as delays are introduced by data retrieval, computation, and data transmission. The control commands for the TB3 are generated on a PC and sent to the robot as quickly as possible to implement strategy~\eqref{u_cont}. Timestamp measurements indicate that a new control command is issued approximately every 0.5 ms, providing a high-frequency approximation of continuous-time behavior. Therefore, for notational convenience, we still refer to this implementation of controller~\eqref{u_cont}  as continuous-time in the sequel. \\
Experimental data and controller parameters are recorded automatically for each trial. Plots were generated using Python’s \texttt{matplotlib} library, and all data was exported in \texttt{.csv} format for further analysis.
\begin{figure}[t]
    \centering
    \includegraphics[width=0.9\linewidth]{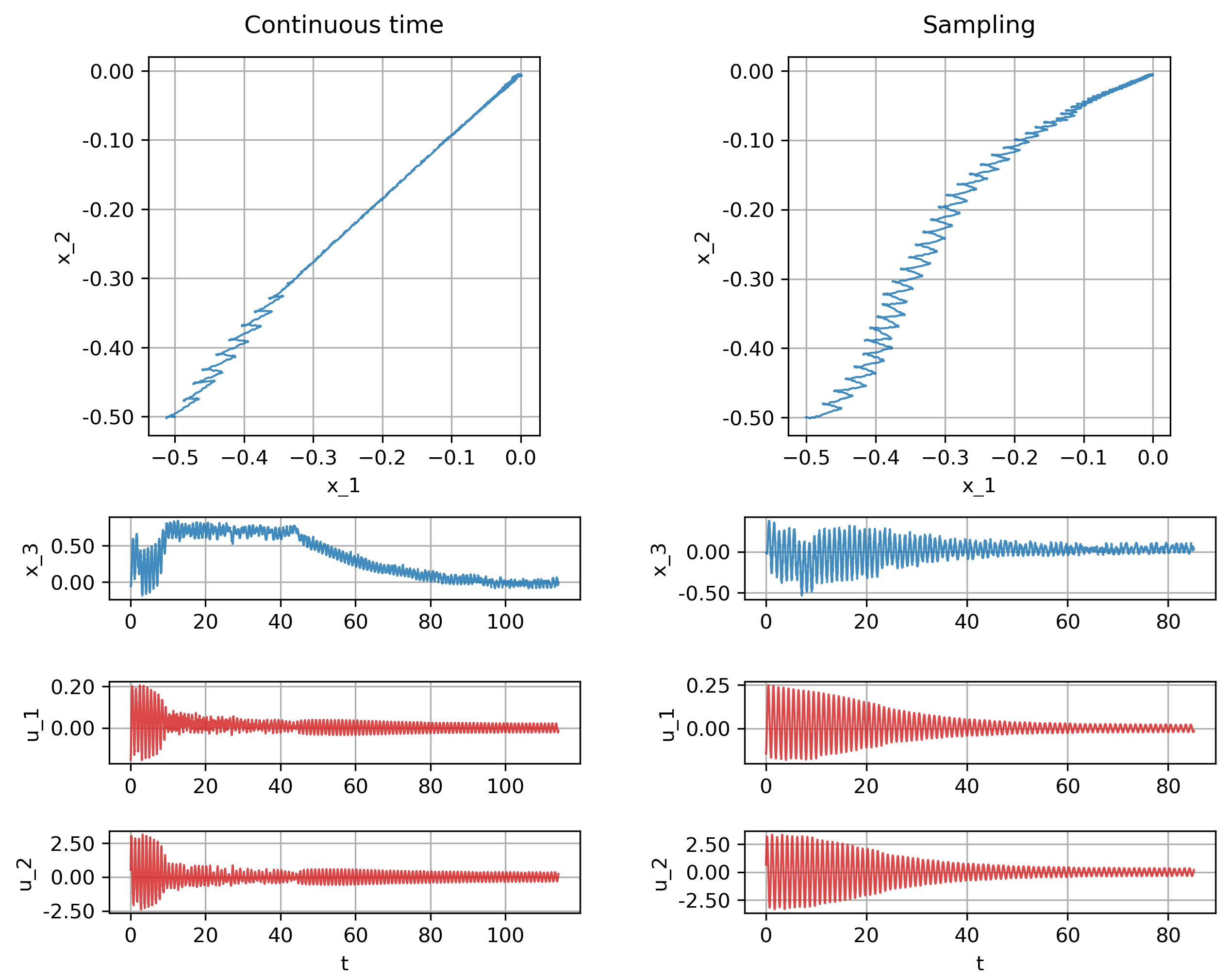}
    \caption{Experiment 1 -- Parameter set $(P1)$: continuous-time vs. sampling.}
    \label{fig:SAAP1_Params_A_real}
\end{figure}

\textbf{Simulation Environment.}
Simulation experiments are conducted using the predefined TB3 Gazebo environment provided in the Robotis e-Manual~\cite{robotis_TB3_gazebo}.
An empty world is used to replicate a simple, obstacle-free arena. The same TB3 model (Burger) as in real world is used to ensure consistency in dimensions, kinematics and dynamics. The model incorporates basic rigid body physics including link masses and joint properties allowing the robot to respond realistically to velocity commands although small inaccuracies in terms of friction and inertia compared to the real-world robot may exist. Due to the software architecture the same Python controller could be used.

\subsection{Experimental Results}
For all experiments, the initial state is set to  $x^0=(-0.5, -0.5, 0)^\top$, and the target state is $x^*=(0, 0, 0)^\top$. The robot stops automatically  if it is in the vicinity of that goal point to obtain comparable convergence time results. More precisely, an orientation angle tolerance of $\Delta x_3 = 0.05~rad$ was chosen for all parameter sets (see (\ref{pset:P1}), (\ref{pset:P2}), (\ref{pset:P3}), (\ref{pset:P4})). While for (P1) and (P2) we chose a target distance tolerance of $\sqrt{x_1^2 + x_2^2} \leq 0.005~m$, for (P3) and (P4) we chose a target distance tolerance of $\sqrt{x_1^2 + x_2^2} \leq 0.03~m$ due to longer convergence times.
 The translational and angular velocities are governed by the control law~\eqref{velControls},~\eqref{aPot} with the frequency parameter $\omega = 2\pi$ (i.e., $\varepsilon = 1$) and the gain $\gamma = 0.05$.
It is important to note that theoretical analyses typically assume sufficiently high control frequencies. However, the robot's physical constraints prevent operation at large frequency ranges.  So, to adjust control bounds to the admissible range $|u_1(t)|\le 0.22$, $|u_2(t)|\le 2.84$, we select relatively small values of $\omega$ and $\gamma$. Moreover, since larger admissible values are allowed for $u_2$, we choose $k_1 < k_2$. This choice avoids the need to further reduce $\omega$ and $\gamma$ while still satisfying the input constraints.\\
For the experimental validation, we use the Lyapunov function candidate~\eqref{V_alpha} and controls of the form~\eqref{velControls},~\eqref{aPot} with one of the following sets of parameters:
\noindent
\begin{minipage}{0.48\textwidth}
\begin{align}
\alpha &= 1,\; k_1 = 0.5,\; k_2 = 8, \tag{P1}\label{pset:P1}\\
\alpha &= 1,\; k_1 = \tfrac{1}{\sqrt{2}},\; k_2 = 4\sqrt{2}, \tag{P2}\label{pset:P2}
\end{align}
\end{minipage}
\hfill
\begin{minipage}{0.48\textwidth}
\begin{align}
\alpha &= 4,\; k_1 = 0.5,\; k_2 = 8, \tag{P3}\label{pset:P3}\\
\alpha &= 10,\; k_1 = 0.5,\; k_2 = 8. \tag{P4}\label{pset:P4}
\end{align}
\end{minipage}

In Experiments~1 and~2, we examine two different parameter selections $(P1)$ and $(P2)$ with the standard sum-of-squares Lyapunov function candidate $V_1(x)=x_1^2 + x_2^2 + x_3^2$.
By comparing the trajectories plots in these two experiment in~Figs.~2 and~3, we observe that the first choice of control coefficients $k_1=0.5$ and $k_2=8$ demonstrates a better  transient performance, so this choice is used in subsequent experiments.

We also compare two different approaches for defining the closed-loop solutions in each of the experiments. For the continuous-time approach~\eqref{u_cont}, the control amplitude vector $a(x)$ is recomputed at each new computation of the controls $u_1$ and $u_2$. For the sampling approach~\eqref{u_sampling}, the coefficients are only updated after a period of $\varepsilon=1$. \\
In Figs.~2--5, the units of $x_1$ and $x_2$ are $m$, $x_3$ is in $rad$,  the translational velocity control $u_1$ is in $m/s$,  and the angular velocity control $u_2$ is in $rad/s$.
Videos of all the experiments are available on YouTube~\cite{YTplaylist}.

\begin{figure}[t]
    \centering
    \includegraphics[width=1\linewidth]{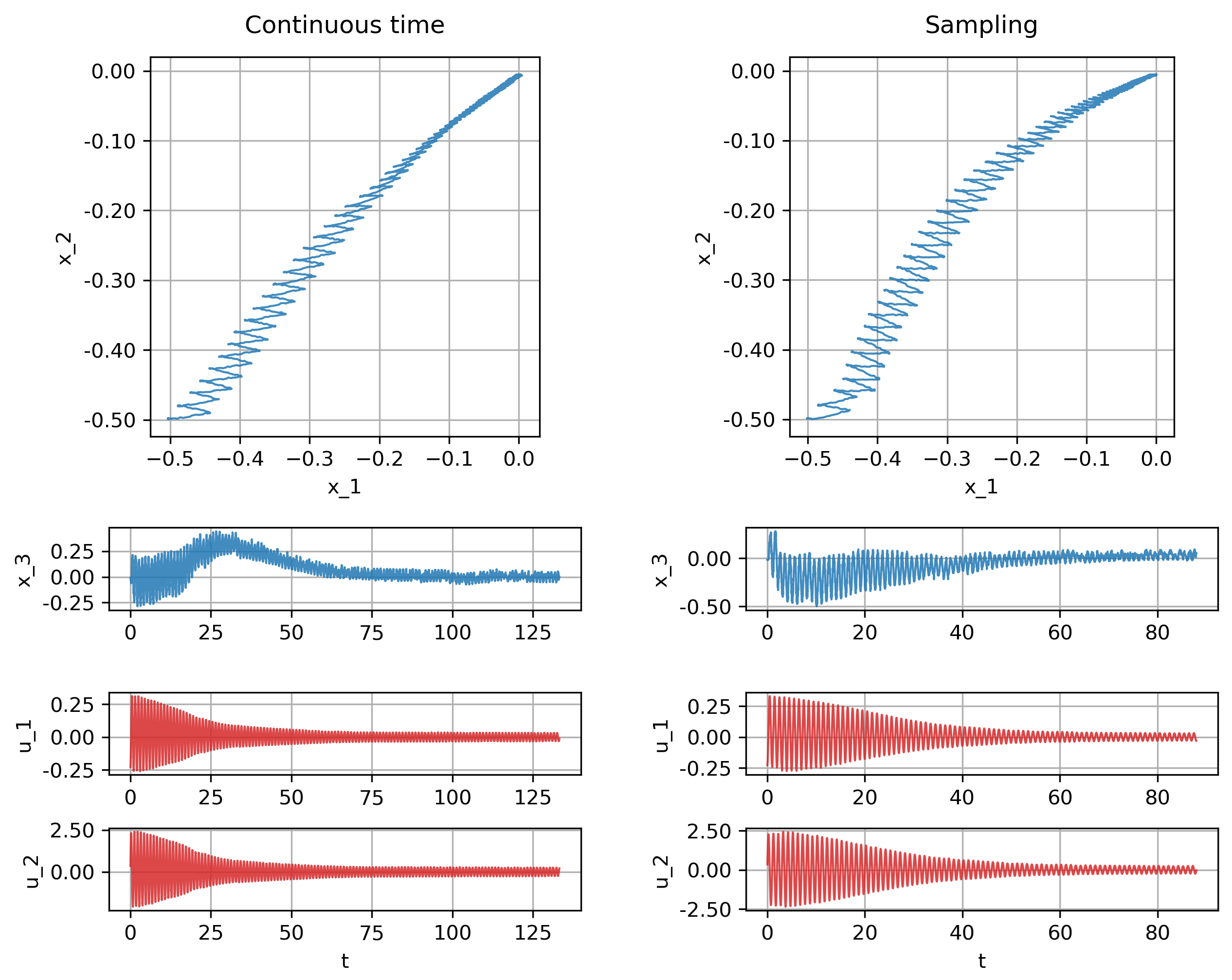}
  \caption{Experiment 2 -- Parameter set $(P2)$: continuous-time vs. sampling.}
    \label{fig:SAAP1_Params_B_real}
\end{figure}
Comparing the continuous-time and sampling approach, we can see that they generate clearly distinct trajectories. The continuous time approach follows a relatively direct and steep path which can especially be seen from the parameter set $(P1)$, where the robot moves straight $\pi/4$ to the goal without chattering from about $1/3$ of the distance. The sampling approach produces a more curved and more chattering but faster path.
When comparing the same approach across parameter sets $(P1)$ and $(P2)$, the overall trajectory shape is preserved, but parameter set $(P2)$ exhibits higher chattering for both approaches, thus $k_1$ and $k_2$ mostly influence chattering.
Across the parameter sets, the two approaches maintain their characteristic differences. Both continuous and sampling approaches have a higher convergence time for parameter set $(P2)$ due to the increased chattering whereas this is more visible for the continuous time approach as the robot was almost not moving in a straight line to the target.

\begin{figure}[t]
    \centering
    \includegraphics[width=1\linewidth]{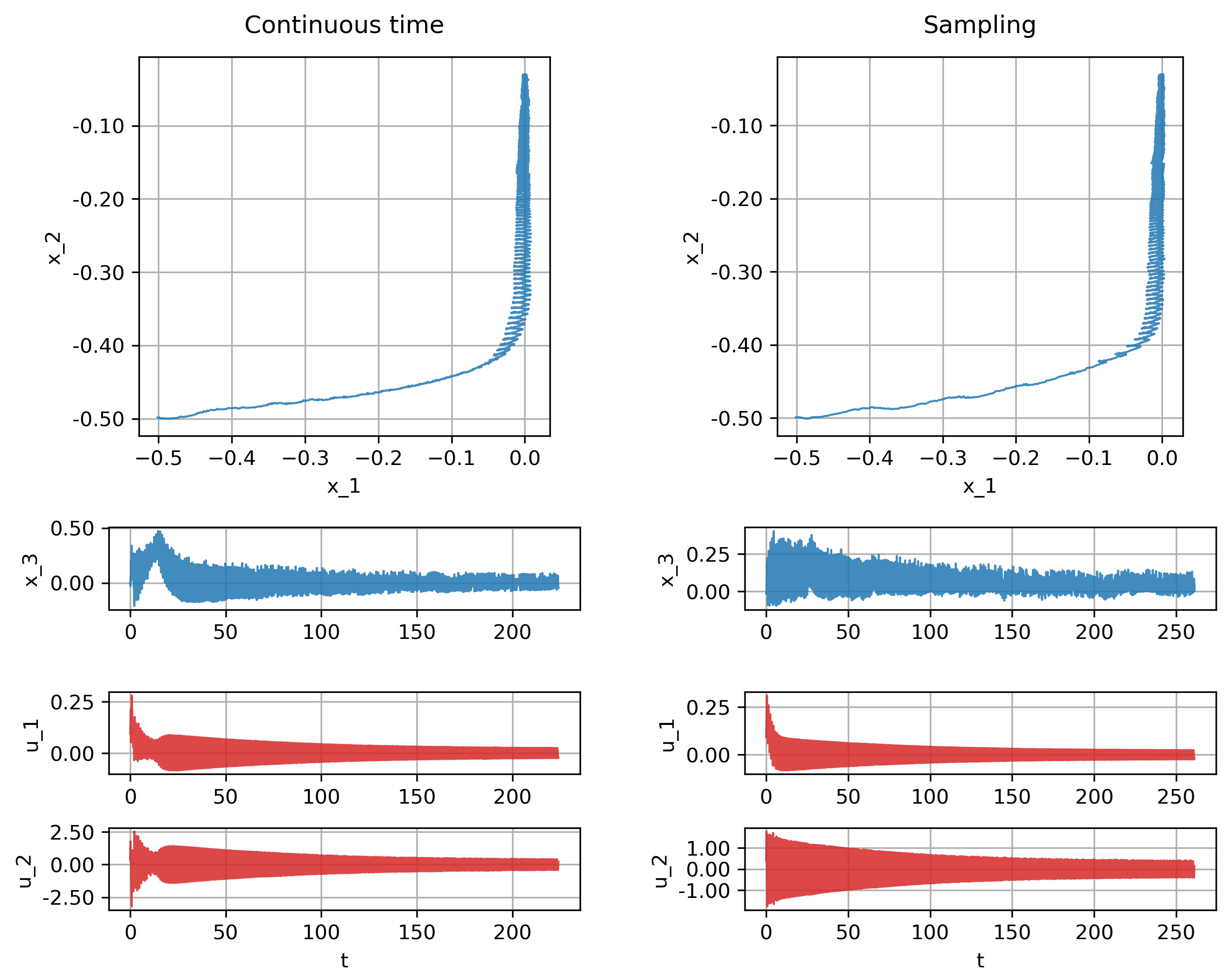}
    \caption{Experiment 3 -- Parameter set $(P3)$: continuous-time vs. sampling.}
    \label{fig:SAAP2_Params_A_real}
\end{figure}

\begin{figure}[h]
    \centering
    \includegraphics[width=1\linewidth]{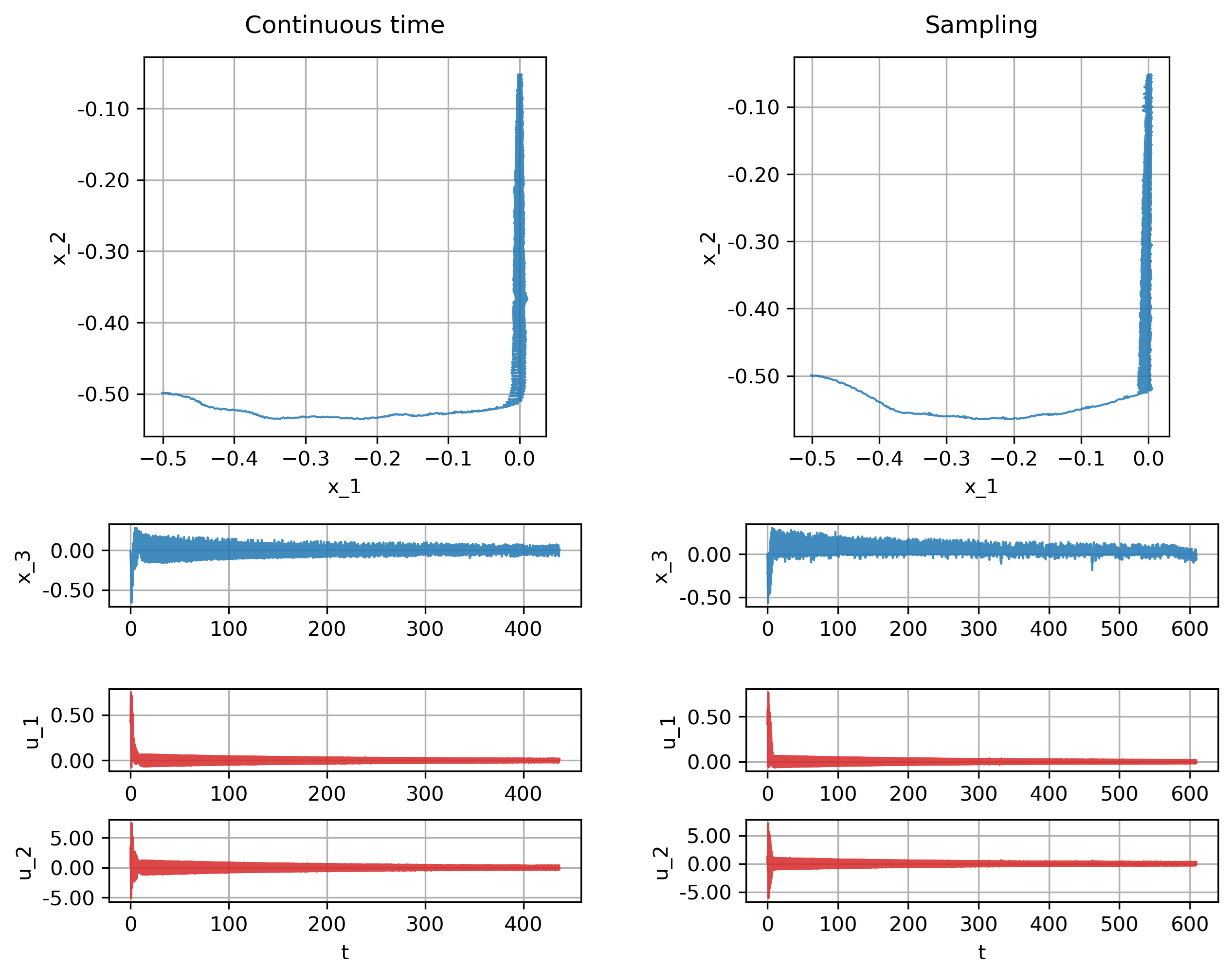}
   \caption{Experiment 4 -- Parameter set $(P4)$: continuous-time vs. sampling.}
    \label{fig:SAAP2_Params_B_real}
\end{figure}
Varying the $\alpha$-parameter in the Lyapunov function candidate~\eqref{V_alpha} drastically changed the shape of the robot's trajectory. In Figs. 4 and 5, we observe a very fast convergence in $x_1$, whereas the overall convergence time increased as the robot had then to move in the non-admissible direction $x_2$, i.e. sideways. Across both parameter sets $(P3)$ and $(P4)$, the sampling approach is slower than the continuous-time one.

We observe in Figs. 2--5 that the control signals $u_1$ and $u_2$, evaluated along the presented trajectories, perfectly satisfy the control constraints specified above, and the transient behavior is acceptable for practical applications.

\section{Conclusions and Future Work}

The main contribution of this work is twofold.
First, the possibility of practical application of the gradient-based controllers for tracking a real mobile robot toward a prescribed target position is justified. It is remarkable to note that although the theoretical stability proof is valid for sufficiently small values of the parameter $\varepsilon$ (and therefore for sufficiently large control frequency $\omega$), our experiments demonstrate the applicability of these controllers at moderate frequencies with $\varepsilon=1$.
The choice of control coefficients $k_1$ and $k_2$
 is justified in the present work based on control constraints for differential-drive robots combined with gradient-based controls under a realistic Lyapunov function candidate.

Second, through the formulation of a suitable optimization problem, an appropriate class of Lyapunov candidates has been studied. The construction of the first terms of a minimizing sequence demonstrates the potential of this approach to reduce chattering behavior and thus minimize zigzagging in the closed-loop dynamics. These observations are well confirmed in Experiments~3 and~4, where the initial stage of the closed-loop trajectory exhibits significantly more regular behavior compared to the strongly zigzagging evolution obtained with a standard sum-of-squares Lyapunov function candidate.

A further research direction could focus on studying the well-posedness of the associated optimization problems and on developing efficient computational schemes for constructing Lyapunov function candidates. We leave this open theoretical question for future investigation.
Another promising direction concerns experimental research aimed at achieving guaranteed-time practical stabilization to reach a prescribed neighborhood of the target point, possibly exploiting non-differentiable Lyapunov function candidates.

\bibliographystyle{spmpsci}

\begin{thebibliography}{10}
\providecommand{\url}[1]{{#1}}
\providecommand{\urlprefix}{URL }
\expandafter\ifx\csname urlstyle\endcsname\relax
  \providecommand{\doi}[1]{DOI~\discretionary{}{}{}#1}\else
  \providecommand{\doi}{DOI~\discretionary{}{}{}\begingroup
  \urlstyle{rm}\Url}\fi

\bibitem{robotis_TB3_overview}
\url{https://emanual.robotis.com/docs/en/platform/turtlebot3/overview/}

\bibitem{robotis_TB3_gazebo}
\url{https://emanual.robotis.com/docs/en/platform/turtlebot3/simulation/#gazebo-simulation}

\bibitem{YTplaylist}
\url{https://www.youtube.com/playlist?list=PLzONPJl2XQ2NXeSbX9ORtu-gSslBxUgxk}

\bibitem{Artstein}
Artstein, Z.: Stabilization with relaxed controls.
\newblock Nonlinear Analysis: Theory, Methods \& Applications \textbf{7}(11),
  1163--1173 (1983)

\bibitem{Bloch15}
Bloch, A.M.: Nonholonomic mechanics and control, vol.~24.
\newblock Springer (2015)

\bibitem{Bro83}
Brockett, R.W.: Asymptotic stability and feedback stabilization.
\newblock Differential Geometric Control Theory pp. 181--191 (1983)

\bibitem{ebrahimi2026impact}
Ebrahimi, A., Azimi, A., Aghdam, M.M.: Impact of numerical integrators on the
  stabilization of holonomic and nonholonomic constraints in the baumgarte
  method.
\newblock Engineering with Computers \textbf{42}(1), 40 (2026)

\bibitem{gao2024fas}
Gao, Y., Zhang, Z., Huang, P., Wu, Y.: Fas-based anti-disturbance stabilization
  control of nonholonomic systems: theory and experiment.
\newblock IEEE Transactions on Automation Science and Engineering  (2024)

\bibitem{GZ18}
Grushkovskaya, V., Zuyev, A.: Obstacle avoidance problem for second degree
  nonholonomic systems.
\newblock In: Proc. 57th IEEE Conf. on Decision and Control, pp. 1500--1505
  (2018)

\bibitem{GZ23}
Grushkovskaya, V., Zuyev, A.: Motion planning and stabilization of nonholonomic
  systems using gradient flow approximations.
\newblock Nonlinear Dynamics \textbf{111}(23), 21,647--21,671 (2023)

\bibitem{miao2025study}
Miao, X., Liu, C.: Study on bicycle stability and self-stabilization mechanism
  by dynamic-and-data-driven surrogate modeling: X. miao, c. liu.
\newblock Nonlinear Dynamics pp. 1--23 (2025)

\bibitem{mondal2020stability}
Mondal, K., Wallace, B., Rodriguez, A.A.: Stability versus maneuverability of
  non-holonomic differential drive mobile robot: Focus on aggressive position
  control applications.
\newblock In: Proc. IEEE Conf. Control Technol. Appl., pp. 388--395 (2020)

\bibitem{pazderski2017waypoint}
Pazderski, D.: Waypoint following for differentially driven wheeled robots with
  limited velocity perturbations: asymptotic and practical stabilization using
  transverse function approach.
\newblock Journal of Intelligent \& Robotic Systems \textbf{85}(3), 553--575
  (2017)

\bibitem{wang2023formation}
Wang, J., Dong, H., Chen, F., Vu, M.T., Shakibjoo, A.D., Mohammadzadeh, A.:
  Formation control of non-holonomic mobile robots: Predictive data-driven
  fuzzy compensator.
\newblock Mathematics \textbf{11}(8), 1804 (2023)

\bibitem{wu2025evolutionary}
Wu, J., Cheng, H., Tian, K., Li, P.: Evolutionary computing control strategy of
  nonholonomic robots with ordinary differential equation kinematics model.
\newblock Electronics \textbf{14}(3), 601 (2025)

\bibitem{zhang2026prescribed}
Zhang, J., Liu, Y., Zheng, Z., Niu, B.: Prescribed-time control for
  nonholonomic systems: a fully actuated systems method.
\newblock Int. J. Syst. Sci. pp. 1--13 (2026)

\bibitem{zhang2022new}
Zhang, Z., Zhang, S., Wu, Y.: New stabilization controller of state-constrained
  nonholonomic systems with disturbances: Theory and experiment.
\newblock IEEE Transactions on Industrial Electronics \textbf{70}(1), 669--677
  (2022)
  
\bibitem{Zu99}
Zuiev, A.: On {B}rockett's condition for smooth stabilization with respect to a
  part of the variables.
\newblock In: 1999 European Control Conference (ECC), pp. 1729--1732. IEEE
  (1999)

\bibitem{ZuSIAM}
Zuyev, A.: Exponential stabilization of nonholonomic systems by means of
  oscillating controls.
\newblock SIAM J. Control Optim. \textbf{54}(3), 1678--1696 (2016)

\bibitem{ZG_CDC25}
Zuyev, A., Grushkovskaya, V.: On classical solutions in the stabilization
  problem for nonholonomic control systems with time-varying feedback laws.
\newblock In: Proc. IEEE 64th Conf. Decis. Control, pp. 7507--7512 (2025)

\end{thebibliography}

\end{document}